\documentclass[a4paper]{article}

\usepackage[english]{babel}

\usepackage[utf8]{inputenc}
\usepackage[utf8]{inputenc}
\usepackage{amsmath}
\usepackage{graphicx}
\usepackage{amssymb}
\usepackage{amsthm}
\usepackage{tikz-cd}
\usepackage{mathrsfs}
\usepackage[colorinlistoftodos]{todonotes}
\usepackage{enumitem}
\usepackage{yfonts}
\usepackage{ dsfont }
\usepackage{MnSymbol}

\title{A new complete Calabi-Yau metric on $\C^3$ }

\author{Yang Li}

\date{\today}
\newtheorem{thm}{Theorem}[section]
\newtheorem{lem}[thm]{Lemma}

\theoremstyle{definition}
\newtheorem{eg}[thm]{Example}

\newtheorem{cor}[thm]{Corollary}

\newtheorem*{rmk}{Remark}
\newtheorem{prop}[thm]{Proposition}
\newtheorem{Def}[thm]{Definition}
\newtheorem{Question}[thm]{Question}
\newtheorem*{Acknowledgement}{Acknowledgement}

\newcommand{\ie}{\emph{i.e.} }
\newcommand{\cf}{\emph{cf.} }

\newcommand{\R}{\mathbb{R}}
\newcommand{\C}{\mathbb{C}}
\newcommand{\Z}{\mathbb{Z}}

\newcommand{\norm}[1]{\left\lVert#1\right\rVert}
\newcommand{\Lap}{\Delta}
\newcommand{\ddbar}{\partial\bar{\partial}}

\DeclareMathOperator{\Tr}{Tr}

\begin{document}
	\maketitle
	
\begin{abstract}
Motivated by the study of collapsing Calabi-Yau 3-folds with a Lefschetz K3 fibration, we construct a complete Calabi-Yau metric on $\C^3$ with maximal volume growth, which in the appropriate scale is expected to model the collapsing metric near the nodal point. This new Calabi-Yau metric has singular tangent cone at infinity $\C^2/\Z_2 \times \C$, and its Riemannian geometry has certain non-standard features near the singularity of the tangent cone, which are more typical of adiabatic limit problems. The proof uses an existence result in H-J. Hein’s PhD thesis to perturb an asymptotic approximate solution into an actual solution, and the main difficulty lies in correcting the slowly decaying error terms. 	
\end{abstract}

\section{Motivations}

This work grows out of the attempt to model the collapsing behaviour of Calabi-Yau metrics on a K3 fibred Calabi-Yau manifold over a Riemann surface, where the K\"ahler class has very small volume on the K3 fibres. Assuming the only singularities in the fibration are nodal, we wish to understand the metric near the critical points of the fibration, in the standard local picture of the ODP on $\C^3$. Knowledge of the model metric is often useful in gluing constructions. The prototype examples are the Kummer construction \cite{Donaldson} and the Gross-Wilson construction \cite{GrossWilson} of CY metrics on K3 surfaces, modelled on the Eguchi-Hanson metric and the Ooguri-Vafa metric, respectively.

The main result of this paper is:

\begin{thm}
There is a complete Calabi-Yau metric of maximal growth on $\C^3$ with the standard holomorphic volume form, whose tangent cone at infinity is the singular cone $\C^2/\Z_2 \times \C$.
\end{thm}

\begin{rmk}
The leading order asymptote and the error estimate will be given in the course of proof (\cf (\ref{asymptoticmetricleadingorder}) and proposition \ref{decayerrorestimate}).
\end{rmk}

We briefly summarize the method.
It is natural to write down the so called semiflat metric as a first approximation, which restricts to the Stenzel metrics on the fibres. This is a singular metric admitting a large symmetry group, suggesting that the construction of Calabi-Yau metric is a cohomogeneity two problem. After appropriate scaling and some modifications to the ansatz, we construct a metric whose volume form is asymptotic to a multiple of the standard volume form on $\C^3$. We carefully examine the error terms involved in the approximate metric, and then apply standard techniques for the Monge-Amp\`ere equation on noncompact manifolds to perturb this into a genuine Calabi-Yau metric.

This Calabi-Yau metric is interesting from several other viewpoints. One context is Gromov-Hausdorff convergence theory in Riemannian geometry. Our metric is a nontrivial Ricci-flat metric on $\C^3$ with maximal volume growth, whose tangent cone at infinity is the singular variety $\C^2/\Z_2 \times \C$; this falsifies a conjecture in the field. The neighbourhood of vanishing cycles in the Lefschetz fibration are asymptotically collapsed to the nodal line, similar to the picture in the unpublished work of Hein and Naber on $A_k$-type singularities. From the viewpoints of Euclidean geometry this is counterintuitive as the collapsed region apparently carries a large amount of Lebesgue measure. The behaviour of the Laplacian is likewise non-standard in this region; in fact the cohomogeneity two problem of inverting the Laplacian effectively reduces to solving ODEs.

Another context in which this metric is expected to arise is related to Joyce's construction of G2 manifolds. This typically involves a G2 orbifold with singularities along an associative submanifold, whose normal direction is modelled on $\C^2/\Z_2$. Joyce and Karigiannis in their unpublished work consider desingularising the orbifold, using a family of Eguchi-Hanson metrics to replace $\C^2/\Z_2$. A more subtle question then is  to understand what happens if some Eguchi-Hanson fibres are allowed to be singular. In the dimensionally reduced case where the associative submanifold is just the product of $S^1$ with a complex curve, it seems plausible that our metric with suitable scaling conventions will model this situation.

It should be pointed out that this work overlaps to some extent with the substantial literature on Calabi-Yau metrics on non-compact manifolds, due to many authors, notably Yau, Tian, Hein, Joyce, Kovalev, Haskins and Nordstr\"om, among others. One closely related work is Joyce's construction of Quasi-ALE Calabi-Yau metrics \cite{Joyce}, since our asymptotic cone is a finite group quotient of $\C^3$. Joyce's construction is more general and combinatorial in spirit, but the decay rate to the asymptotic flat metric is much faster in his case, and we do not use his result in a direct way. The announced work of Hein and Naber about $A_k$-type singularities contains many common geometric features, and may use similar strategies. In our construction of Calabi-Yau metric, the work we use substantially is Hein's PhD thesis \cite{Hein2}, which gives a criterion for perturbing an approximate Calabi-Yau metric into a genuine one.

It seems likely to the author that there are higher dimensional analogues of our construction for the standard Lefschetz fibration of $\C^n$ over $\C$. The difficulty comes from the computational complexity of precisely understanding the main error terms in the ansatz. Some informal discussions on the moduli question of classifying complete Calabi-Yau metrics on $\C^n$ with maximal growth will be given at the end.

\begin{Acknowledgement}
The author is grateful to his PhD supervisor Simon Donaldson and co-supervisor Mark Haskins for their inspirations, suggestions and encouragements, and to the London School of Geometry and Number Theory (Imperial College London, UCL and KCL) for providing a stimulating research environment. He would also like to thank H-J. Hein, V. Tosatti and G. Szekelyhidi for comments.

This work was supported by the Engineering and Physical Sciences Research Council [EP/L015234/1], the EPSRC Centre for Doctoral Training in Geometry and Number Theory (The London School of Geometry and Number Theory), University College London. The author is also funded by Imperial College London for his PhD studies.
\end{Acknowledgement}



\section{The Stenzel metric}


We review the Stenzel construction of Ricci-flat metrics on the smoothings of ODP singularities, to set up some notations and to derive some useful formulae. Let $n\geq 3$. We start with the standard quadratic function 
\[
f=\sum_1^n z_i^2: \C^n\to \C,
\]
and consider the fibres $X_y=\{ \sum {z_i^2}=y    \}.$ For later reference, let us establish a few conventions. The standard Euclidean metric on $\C^n$ is $\omega_E=\frac{1}{2}\sqrt{-1}\sum dz_i d\bar{z_i}$; we will sometimes also denote by $\omega_E$ its restriction to fibres. The volume form is $vol_E=\frac{1}{n!}\omega_E^n=\frac{1}{2^n}\widetilde{vol_E}$, where the normalised volume form is $ \widetilde{vol_E}=\prod_i (\sqrt{-1} dz_i d\bar{z_i}  )$. The pullback of the standard metric from the base $\C$ is $\omega_0=\frac{1}{2}\sqrt{-1}df d\bar{f}$. The norm-squared function on $\C^n$ is $H=\sum |z_i|^2$. The holomorphic $n$-form on $\C^n$ is $\Omega=dz_1\ldots dz_n$.

We consider the ansatz $\phi=F(H)$ for the K\"ahler potential on $X_y$; this amounts to imposing $SO(n, \R)$ symmetry. We compute the metric
\[
\bar{\partial}\phi=F'(H)\bar{\partial} {H},
\]
\[
\sqrt{-1}\partial\bar{\partial} \phi=F'\sqrt{-1} \partial\bar{\partial}H+F''\sqrt{-1} \partial{H}\wedge \bar{\partial} H.
\]
Using the binomial theorem, we have
\[
(\sqrt{-1}\ddbar{\phi})^{n-1}=F'^{n-1} (\sqrt{-1}\ddbar{H})^{n-1}
+(n-1)F'^{n-2}F'' (\sqrt{-1}\ddbar{H})^{n-2}\sqrt{-1} \partial H\bar{\partial}H.
\]
To understand the terms, we compute
\[
\frac{1}{4} (\sqrt{-1}\ddbar{H})^{n-1}\sqrt{-1}df d\bar{f}=(n-1)! H\widetilde{vol_E}.
\]
\[
\frac{1}{4} (\sqrt{-1}\ddbar{H})^{n-2}\sqrt{-1} \partial H\bar{\partial}H\sqrt{-1}df d\bar{f}=(n-2)! (H^2-|y|^2) \widetilde{vol_E}.
\]
Thus
\begin{equation}\label{fibrevolumeform}
\frac{1}{4} (\sqrt{-1}\ddbar{\phi})^{n-1}\sqrt{-1}df d\bar{f} =(n-1)! \widetilde{vol_E} \{   HF'^{n-1} +(H^2-|y|^2) F'^{n-2}F''     \}.
\end{equation}

By the adjunction formula, we can write down the holomorphic volume form $\Omega_y$ on $X_y$ by $\Omega=df\wedge \Omega_y$. This gives
\begin{equation}\label{fibreCalabiYauvolumeform}
\widetilde{vol_E}=i^{n^2} \Omega\wedge \overline{\Omega}=\sqrt{-1}df d\bar{f}\wedge (i^{(n-1)^2} \Omega_y\wedge \overline{\Omega_y}).
\end{equation}

\subsection{The homogeneous Monge-Amp\`ere equation}

The homogeneous  Monge-Amp\`ere equation on $X_y$ is
\begin{equation}
(\sqrt{-1}\ddbar{\phi})^{n-1}=0.
\end{equation}
By (\ref{fibrevolumeform}), this can be written as
\[
 HF'^{n-1} +(H^2-|y|^2) F'^{n-2}F''=0
\]
We can integrate this to obtain
\[
F'=\frac{\text{Const} }{ (  H^2-|y|^2   ) ^{1/2}    }.
\]
By a change of variable 
$
H=|f|\cosh \tau,
$
We obtain the solution up to an affine transformation
\[
\phi=F(H)= \tau.
\]

\subsection{The inhomogeneous Monge-Amp\`ere equation} 

The equation defining the Calabi-Yau metric on $X_y$ is
\begin{equation}
(\sqrt{-1}\ddbar{\phi})^{n-1}=i^{n^2} \Omega_y\wedge \overline{\Omega_y}.
\end{equation}
We can wedge both sides by $\sqrt{-1}df d\bar{f}$, and use (\ref{fibrevolumeform}) with (\ref{fibreCalabiYauvolumeform}) to see the following identity holds in $\Omega^{n,n}_{\C^n}|_{X_y}$:
\[
4(n-1)! \{   HF'^{n-1} +(H^2-|y|^2) F'^{n-2}F''     \}=1.
\]
We can simplify the equation further by a scaling observation: let $F=|y|^{\frac{n-2}{n-1} }\tilde{F}(\frac{H}{|y|})$, then $\tilde{F}$ satisfies the ODE
\[
x (\frac{d\tilde{F}}{dx} )^{n-1}+(x^2-1) (\frac{d \tilde{F} }{dx})^{n-2} \frac{d^2 \tilde{F} }{dx^2}=c_1=\frac{1}{4(n-1)!}.
\]
Now make the change of variable $x=\frac{H}{|y|}=\cosh \tau$. Then the equation becomes
\[
\frac{d}{d\tau} [  ( \frac{d \tilde{F} }{d\tau}    )   ^{n-1}          ]=c_1(n-1) \sinh^{n-2}\tau.
\]
If we impose the initial condition $\frac{d \tilde{F} }{d\tau}(0)=0$, then the solution defines up to a numerical factor the Stenzel metric.

\begin{eg}
For $n=3$, the above equation implies $\frac{d \tilde{F}} {d\tau}=\frac{1}{\sqrt{2} } \sinh (\frac{\tau}{2 })$,
$\tilde{F}=\sqrt{2} \cosh(\frac{\tau}{2 }    )=\sqrt{ \frac{H}{|y|} +1  }$. Thus the K\"ahler potential is $\phi=F(H)=\sqrt{H+|y|}$.
\end{eg}

\section{A general symmetric ansatz}

We consider K\"ahler potentials on $\C^n$ of the form 
$
\phi=F(H, \eta)
$ where $\eta=|f|^2$. This amounts to imposing the symmetry under $SO(n,\R)\times U(1)$. Notice that this is not a cohomgeneity-one problem, so it may be hard to write down explicit solutions to interesting equations. We shall nevertheless exhibit solutions whose asymptotic behaviour resembles a Calabi-Yau metric. This symmetry reduction technique is very common, and in particular used by Hein and Naber in their work on $A_k$-type singularities.

We imitate the computations about the Stenzel metric, but now we need to keep track of all partial derivatives.
\[
\bar{\partial}\phi=F_H \bar{\partial}{H}+F_\eta f d\bar{f},
\]
\[
\begin{split}
\sqrt{-1}\partial\bar{\partial} \phi=& F_H\sqrt{-1} \partial\bar{\partial}H+F_{HH}\sqrt{-1} \partial{H}\wedge \bar{\partial} H
+\sqrt{-1} F_{H\eta} (\bar{f} df\wedge \bar{\partial}H +f \partial{H}  \wedge d\bar{f}         )\\
&+
(\eta F_{\eta\eta}  +F_\eta )\sqrt{-1}df d\bar{f}   .
\end{split}
\]
Now using the binomial theorem, one can expand
\[
\begin{split}
(\sqrt{-1}\partial\bar{\partial} \phi )^n=& (F_H \sqrt{-1} \partial\bar{\partial}H + F_{HH}\sqrt{-1} \partial{H}\wedge \bar{\partial} H  )^n \\
&+
n ( F_H \sqrt{-1} \partial\bar{\partial}H + F_{HH}\sqrt{-1} \partial{H}\wedge \bar{\partial} H               )^{n-1}  \\
&\{  ( \eta F_{\eta\eta}  +F_\eta     ) \sqrt{-1}df d\bar{f} +  \sqrt{-1} F_{H\eta} (\bar{f} df\wedge \bar{\partial}H +f \partial{H}  \wedge d\bar{f}         )          \}\\
&+ {n\choose 2} (    F_H \sqrt{-1} \partial\bar{\partial}H + F_{HH}\sqrt{-1} \partial{H}\wedge \bar{\partial} H             )^{n-2}  (  2\eta F_{H\eta}^2 df d\bar{f} \wedge \partial H \bar{\partial}H                )         .
\end{split}
\]
Now we compute individual terms
\[
\begin{split}
(F_H \sqrt{-1} \partial\bar{\partial}H + F_{HH}\sqrt{-1} \partial{H}\wedge \bar{\partial} H  )^n=& F_H^n n!\widetilde{vol_E} +n F_H^{n-1} F_{HH} H (n-1)! \widetilde{vol_E} \\
=& n! \widetilde{vol_E} \{  F_H^n+   F_H^{n-1} F_{HH}  H     \}.
\end{split}
\]
\[
\begin{split}
( F_H \sqrt{-1} \partial\bar{\partial}H + F_{HH}\sqrt{-1} \partial{H}\wedge \bar{\partial} H               )^{n-1} \sqrt{-1}df d\bar{f}= 4(n-1)! \widetilde{vol_E} \\ \{ HF_H^{n-1}+(H^2-\eta) F_H^{n-2} F_{HH}              \}.
\end{split}
\]

\[
\begin{split}
&( F_H \sqrt{-1} \partial\bar{\partial}H + F_{HH}\sqrt{-1} \partial{H}\wedge \bar{\partial} H               )^{n-1} \sqrt{-1}(   \bar{f} df\wedge \bar{\partial}H +f \partial{H}  \wedge d\bar{f}                  )\\&= 
F_H^{n-1 }(\sqrt{-1} \partial\bar{\partial}H          )^{n-1} \sqrt{-1}(   \bar{f} df\wedge \bar{\partial}H +f \partial{H}  \wedge d\bar{f}                  )
=
4\eta F_H^{n-1 } (n-1)! \widetilde{vol_E} .
\end{split}
\]

\[
\begin{split}
&( F_H \sqrt{-1} \partial\bar{\partial}H + F_{HH}\sqrt{-1} \partial{H}\wedge \bar{\partial} H               )^{n-2} df d\bar{f} \wedge \partial H \bar{\partial}H  
\\&= 
F_H^{n-2 }(\sqrt{-1} \partial\bar{\partial}H          )^{n-2} df d\bar{f} \wedge \partial H \bar{\partial}H  
=
-4(n-2)! F_H^{n-2 } (H^2-\eta ) \widetilde{vol_E} .
\end{split}
\]

Combining these computations, we arrive at a formula for the volume form:
\begin{prop}
For a symmetric K\"ahler potential $\phi=F(H,\eta)$, we have a formula
\begin{equation}\label{symmetricansatz}
\begin{split}
(\sqrt{-1}\partial\bar{\partial} \phi )^n=& n! \widetilde{vol_E} \{
\{
 F_H^n+   F_H^{n-1} F_{HH}  H  
\}\\
&+
4(  \eta F_{\eta\eta}  +F_\eta     )
\{
HF_H^{n-1}+(H^2-\eta) F_H^{n-2} F_{HH}    
\}\\
&+4F_{H\eta}\eta F_H^{n-1}  
-4F_{H\eta}^2 \eta F_H^{n-2} (H^2-\eta)
\}.
\end{split}
\end{equation}
\end{prop}

\begin{eg}
If the potential does not depend on $\eta$, then 
\[
(\sqrt{-1}\partial\bar{\partial} \phi )^n=n! \widetilde{vol_E} \{
F_H^n+   F_H^{n-1} F_{HH}  H  
\}.
\]
For example, for $F=\frac{1}{2}H$, one recovers the Euclidean metric and its volume form. For $F=\log H$, one recovers the homogeneous Monge-Amp\`ere equation.
\end{eg}

\begin{eg}
Let $n=3$. Let us specialise the ansatz to the form \begin{equation}
\phi=F(H,\eta)=\frac{1}{2}\eta+ \sqrt{H+a(\eta) }
\end{equation}
defined by a non-negative smooth function $a(\eta)$, with the property that $\sqrt{-1}\ddbar{\phi}$ is positive, and $a(\eta)$ is asymptotic to $\eta^{1/2}$ for large $\eta$. An example of this is $a(\eta)=\sqrt{\eta+1}$. 

We compute the various quantities involved in the general formula (\ref{symmetricansatz}).
\[
F_H=\frac{1}{2}\frac{1}{\sqrt{ H+a  } }, \quad F_{HH}=\frac{-1}{4}\frac{1}{(\sqrt{ H+a  })^3 },
\]
\[
HF_H^2+(H^2-a^2)F_H F_{HH}=\frac{1}{8},\quad (a^2-\eta)F_H F_{HH}=-\frac{1}{8} \frac{1}{(H+a)^2}(a^2-\eta),
\]
\[
F_H^3+F_H^2 HF_{HH}=\frac{H+2a}{16} \frac{1}{ (\sqrt{ H+a  })^5      },
\]
\[
F_\eta=\frac{1}{2}+\frac{1}{2 \sqrt{H+a}  } a', \quad F_{\eta\eta}= \frac{a''}{ 2(\sqrt{ H+a  })      }- \frac{a'^2}{ 4(\sqrt{ H+a  })^3      },
\]
\[
F_\eta+\eta F_{\eta\eta}=\frac{a'+\eta a''}{ 2(\sqrt{ H+a  })      }- \frac{\eta a'^2}{ 4(\sqrt{ H+a  })^3      },
\quad
F_{H\eta}=-\frac{1}{4} \frac{a'}{ (\sqrt{ H+a  }) ^{3}     }.
\]
Combining the computations, we get the volume form of the ansatz
\begin{equation}
(\sqrt{-1}\ddbar{\phi})^3= \frac{3}{2}\widetilde{vol_E} \{1+ (a'+\eta a'')\frac{1}{ \sqrt{ H+a  }     } +O(\frac{1}{(\sqrt{ H+a  })^3    }  ) \}.
\end{equation}
We comment that $a'+\eta a''\sim \frac{1}{4\eta^{1/2}}$ for large $\eta$. In particular, to the leading order the volume measure is a constant multiple of the Lebesgue measure. But the leading order error is quite large.

\begin{eg}(The main ansatz)
We modify the above ansatz to the form \begin{equation}\label{smoothingexample2scaled}
\phi=F(H,\eta)=\frac{1}{2}\eta+ \sqrt{H+a(\eta, H) },\quad a(\eta, H)=\sqrt{\eta+ \sqrt{H+1} }. 
\end{equation}
This is the main ansatz we shall use in the rest of the paper.
The computation follows very similar lines but is technically more complicated, and for the moment we shall just give a relatively crude asymptotic expression
\begin{equation}\label{asymptoticvolume1}
\begin{split}
(\sqrt{-1}\ddbar{\phi})^3&= \frac{3}{2}\widetilde{vol_E} \{1+ (a_\eta+\eta a_{\eta\eta})\frac{1}{ \sqrt{ H+a  }      }+(a_H+2Ha_{HH}) +O(\frac{1}{H^{3/2} }) +O(\frac{1}{a^3 }    ) \} \\
&=\frac{3}{2}\widetilde{vol_E} \{1+\frac{1}{4 H^{1/2} a    } +O(\frac{1}{H^{3/2}    }  )+O(\frac{1}{a^3 }    ) \}
.
\end{split}
\end{equation}
For small values of $\eta$ there is some improvement in the decay rate, compared to the previous example.
\end{eg}

\end{eg}

\section{Relation to the collapsing metric problem}\label{collapsingproblem}

We shall now explain the origin of the main ansatz (\ref{smoothingexample2scaled}).

\textbf{Motivating question}: Consider a Calabi-Yau 3-fold which admits a Lefschetz fibration over a Riemann surface. If we vary the K\"ahler class so that the fibres have area of order $t<<1$, we obtain some highly collapsed Calabi-Yau metric on the total space. Can we construct a local model on $\C^3$ describing what happens near the ODP points in the fibration?

Let $n=3$.
The `semiflat form' is a closed (1,1) form $\sqrt{-1}\ddbar{\phi}$ whose restriction to each fibre agrees with the Calabi-Yau metric on the fibre. This involves some ambiguity because one can always adjust $\phi$ by a potential pulled back from the base. Motivated by the above question, we can na\"ively write down the following ansatz, as a first approximation to the collapsed Calabi-Yau metric, in the local region described by the standard model of the Lefschetz fibration:
\begin{equation}
\phi=F(H,\eta)=\frac{1}{2}\eta+ t \sqrt{H+\eta^{1/2} }.
\end{equation}
Here $\frac{1}{2}\eta$ gives $\omega_0$, the pullback of the standard metric on the base, and the fibrewise restriction is just the Stenzel metric for $n=3$, rescaled to size $t<<1$.

This na\"ive ansatz is not smooth at the central fibre $\eta=0$. But at least we know $\phi$ is PSH, by the following easy and well known result:
\begin{lem}
If $\psi$, $\psi'$ are PSH, then $\log (e^\psi+e^{\psi'})$ is PSH.
\end{lem}

We shall use the formula (\ref{symmetricansatz}) as a test for the validity of this na\"ive ansatz as an approximate model for a Calabi-Yau metric. The semiflat condition gives 
\[
HF_H^{n-1}+(H^2-\eta) F_H^{n-2} F_{HH}=c_1t^{2}=\frac{1}{8}t^2.
\]
To leading order, we expect $ \eta F_{\eta\eta}  +F_\eta \sim \frac{1}{2}$, by ignoring all terms with $t$ dependence. We notice that all terms with $H$ dependence necessarily also carry a $t$ factor, so the correction terms $ F_H^n+   F_H^{n-1} F_{HH}  H $ and $4F_{H\eta}\eta F_H^{n-1}  
-4F_{H\eta}^2 \eta F_H^{n-2} (H^2-\eta)$ are both of order $t^{3}$. This means the leading order term is $(\sqrt{-1}\partial\bar{\partial} \phi )^3\sim \frac{3}{2} t^2 \widetilde{vol_E}$, and this metric is approximately Calabi-Yau, with error terms typically suppressed by a factor $t$.

 Now $F_H\sim \frac{t}{H^{1/2} }$, so at the scale $H \sim t^{2/3}$, we have ${F_H}^{3}\sim t^2$, which is comparable to the leading order volume form, and the approximation breaks down. 
 
 Similarly $F_\eta-\frac{1}{2}\sim \frac{t}{H^{1/2} \eta^{1/2 } }$, so if $H\eta\sim t^2$, the approximation fails. 
 
 The source of the first failure has an interesting geometric interpretation. Think about the fibred Calabi-Yau 3-fold with the collapsing CY metric which we are trying to model. At the nodal point, \ie the origin, the CY metric should be locally Euclidean, but the volume form has order $t^2$. Therefore it seems reasonable to expect that near the origin, this collapsing CY metric is uniformly equivalent to $t^{2/3}\omega_E$. This uniform equivalence is expected to hold up to the scale $H\sim t^{2/3}$, at which the Calabi-Yau metric matches up with the semiflat approximation ansatz. We may call the scale $H\sim t^{2/3}$ the `quantisation scale'. The reason for choosing this terminology is that, when $H$ is larger than this scale, the CY metric approximates a rescaled semiflat metric. This picture is formally similar to the semiclassical limit in quantum mechanics. The quantisation scale is where this approximation breaks down, which can be imagined as quantum fluctuation effects.

The second failure due to the smallness of $\eta$ seems to be an artefact of our non-smooth choice of ansatz. There is no canonical choice for a smoothing. We make the following choice to achieve good decay properties:
\begin{equation}\label{smoothingexample2}
\phi=\frac{1}{2}\eta+ t \sqrt{H+(\eta+t\sqrt{H+t^{2/3} })^{1/2} }.
\end{equation}
The previous discussions essentially apply to the modified situation, with $H\eta\sim t^2$ replaced by $H(\eta+t^{4/3})\sim t^2$,
and now the sole source of failure is due to $H\sim t^{2/3}$.

We zoom in to the quantisation scale. This amounts to using the coordinates $\frac{1}{t^{1/3}}z_i$ instead of $z_i$, and scale by a factor $t^{-4/3}$.  Starting from (\ref{smoothingexample2}), we get
\begin{equation}
\phi=\frac{1}{2}\eta+ \sqrt{H+(\eta+\sqrt{H+1} )^{1/2} }.
\end{equation}
which is our main ansatz.

\section{Riemannian geometry of the main ansatz (\ref{smoothingexample2scaled})}

We study the asymptotic properties of the metric ansatz (\ref{smoothingexample2scaled}). As a preliminary remark, the metric has asymptotically the Euclidean volume form. At one stage we will make use of a non-holomorphic parametrisation, so we shall distinguish the K\"ahler form $\omega$ and the associated Riemannian metric tensor $g$.

\subsection{Asymptotic metric}

Now we study the properties of the metric for $H>>1$, only to the leading order. The asymptotic expression is
\begin{equation}\label{asymptoticmetricleadingorder}
\begin{split}
\omega=\sqrt{-1}\ddbar{\phi}\sim \frac{1}{2}df\wedge d\bar{f}+ \frac{\sqrt{-1} }{ 2\sqrt{H+a}    }\ddbar{H}- \frac{\sqrt{-1} }{ 4 (\sqrt{H+a})^3    }\partial{H}\wedge \bar{\partial} H  \\
\sim \frac{1}{2}df\wedge d\bar{f}+ \frac{\sqrt{-1} }{ 2\sqrt{H+\eta^{1/2}}    }\ddbar{H}- \frac{\sqrt{-1} }{ 4 (\sqrt{H+\eta^{1/2}})^3    }\partial{H}\wedge \bar{\partial} H
\end{split}
\end{equation}

We want to understand the qualitative features of the distance function. 
\begin{lem}
In the region $P\in \{H>1\}$, there is a uniform equivalence
\begin{equation}
C ( |f|+H^{1/4}     )\leq d(0,P) \leq C'( |f|+H^{1/4}   ).
\end{equation}
In other words, $d(0,P)$ is uniformly equivalent to the function $a$ for $P\in \{H>1\}$.
\end{lem}

\begin{proof}
Using the asymptotic formula, one sees
$
C\omega \geq \sqrt{-1}df\wedge d\bar{f}, 
$
so $|f|\leq Cd(0, P)$. Similarly one sees for $H>1$,
$C\omega \geq \frac{\sqrt{-1} }{  \sqrt{H}   }\ddbar{H}$, so along any path parametrised by the arclength, \[\frac{d|z|}{ds} \leq C|z|^{1/2},\] which integrates to give $H^{1/4} \leq C d(0,P)$.

Now for the reverse direction, we take the radial path joining the origin and the point $P$, and compute the length. This gives 
$
d(0,P) \leq C( |f|+H^{1/4}   ).
$
\end{proof}

\begin{lem}
Let $r>>1$. We have the maximal volume growth condition:
\begin{equation}
Cr^6 \leq Vol( B_g(0,r )    ) \leq C'r^6.
\end{equation}
\end{lem}

\begin{proof}
The description of the distance function implies that
\[
\{ H^{1/4}\leq Cr  \}\cap \{|f|\leq Cr \}   \subset   B_g(0,r)\subset  \{ H^{1/4}\leq C'r  \}\cap \{|f|\leq C'r \} \]
for some appropriate constants $C$ and $C'$. Thus we can consider instead the set $\{ H\leq r^4  \}\cap \{|f|\leq r \}$ and estimate its volume. This is an exercise in Lebesgue integration theory. 

Since the asymptotic volume measure is comparable to the Lebesgue measure,
\[
Vol_{\omega}( \{ H\leq r^4  \}\cap \{|f|\leq r \}            )  \\
 \sim Vol_{Lebesgue} ( \{ H\leq r^4  \}\cap \{|f|\leq r \}            ).
\]
By the homogeneity of the Lebesgue measure, the RHS is
\[
r^{12} Vol_{Lebesgue} ( \{ H\leq 1 \}\cap \{|f|\leq r^{-3} \}            ) .
\]
Now for very large $r$, the set $\{ H\leq 1 \}\cap \{|f|\leq r^{-3} \}$ is approximately the product of $ \{ H\leq 1 \}\cap X_0 $ with the disc $\{ |f|\leq r^{-3}  \}\subset \C $, so the above
\[
 \sim r^{12} Area ( \{ H\leq 1 \}\cap X_0      ) Area(  \{ |f|\leq r^{-3}  \}\subset \C      )\sim r^6.
\]
Here the area of $\{ H\leq 1 \}\cap X_0 $ is induced from the restriction of the Euclidean metric. The claimed result follows. 
\end{proof}

We define a family of cones in $\C^3$ by $E_\epsilon=\{ |f|\geq \epsilon H         \}$, for $0<\epsilon<1$. In practice $\epsilon$ is a fixed small number. Our next result says these cones are highly squashed in the metric $g$.

\begin{lem}\label{squash}
Take $r>1$.
Let $P\in B_g( r ) \cap E_{\epsilon}$, then the distance of $P$ to the complement of $E_\epsilon$ has an upper bound
\begin{equation}
d_g(P, \C^3 \setminus E_\epsilon) \leq \frac{C}{\epsilon^{1/4}}r^{1/4}.
\end{equation}
\end{lem}

\begin{proof}
We start by observing that on $\{f=1\}$ with the Stenzel metric, there is a uniform diameter bound $\text{diam}(\{ H\leq \frac{1}{\epsilon}  \}  ) \leq \frac{C}{\epsilon^{1/4}}$, for $0<\epsilon<1$. By scaling, on the fibre $X_f$ with the Stenzel metric, the diameter bound reads
\[
\text{diam}(\{ H\leq \frac{|f|}{\epsilon}  \} \subset X_f ) \leq \frac{C}{\epsilon^{1/4}}|f|^{1/4}  .
\]
Now if $|f|\geq 1$, then the metric $g$ restricted to the fibre $X_f$ is uniformly equivalent to the Stenzel metric, because $g$ is by construction essentially the semiflat approximation, for large values of $\eta$. Thus if $P\in X_f \cap E_\epsilon$, we can move the point inside $X_f$ and exit the region $E_\epsilon$ within distance $\frac{C}{\epsilon^{1/4}}(|f|^{1/4}+1)\leq \frac{C}{\epsilon^{1/4}}(r^{1/4}+1)$. Here we add in the constant 1 to make the argument work also for $|f|\leq 1$. This means $d(P, \C^3 \setminus E_\epsilon) \leq \frac{C}{\epsilon^{1/4}}(r^{1/4}+1)\leq \frac{C}{\epsilon^{1/4}}r^{1/4}$.
\end{proof}

\begin{rmk}
If we normalise the metric by a factor $\frac{1}{r^2}$, then $P\in B_{\frac{1}{r^2}g  } (1)$, and this estimate says $d_{ \frac{1}{r^2}g  } (P, \C^3\setminus E_{\epsilon}) \leq \frac{C}{\epsilon^{1/4}r^{3/4} }$. Intuitively, for large $r$ and fixed $\epsilon$, this means the region $E_\epsilon$ has very small solid angle.
\end{rmk}

Now we concentrate on the region $ (\C^3\setminus E_\epsilon) \cap \{H>>1  \}$, with $\epsilon$ a fixed small number. In this region, up to $O(\epsilon)$ error, the asymptote of $\omega$ as $H\to \infty$ simplifies further to
\[
\omega \sim \frac{1}{2}\sqrt{-1}df\wedge d\bar{f}+ \frac{\sqrt{-1} }{ 2\sqrt{H}    }\ddbar{H}- \frac{\sqrt{-1} }{ 4 (\sqrt{H})^3    }\partial{H}\wedge \bar{\partial} H.
\]

Ehresmann's theorem suggests the fibration $f$ restricted to the region $(\C^3\setminus E_\epsilon) \cap \{H>>1  \}$ is a smooth fibration over $\C$. For example, if we calculate the monodromy along the standard vanishing paths, we can obtain an explicit trivialisation:

\begin{lem}
We have a diffeomorphism $\C^3\setminus \{ H=|f| \}  \simeq (X_0\setminus \{0\} )\times \C $ over  $\C  $, given by
\[
z\in \C^3\setminus \{ H=|f| \} \mapsto (w,f)\in (X_0\setminus \{0\})\times \C,
\]
\begin{equation}
w=\frac{1}{2}( {(\frac{H-|f|}{H+|f|})      }^{1/4} +  {(\frac{H+|f|}{H-|f|})      }^{1/4}           )z+ \frac{1}{2}( {(\frac{H-|f|}{H+|f|})      }^{1/4} -  {(\frac{H+|f|}{H-|f|})      }^{1/4}           )\frac{f\bar{z}} {|f|}.
\end{equation}
The inverse map can be described as follows:
\[
z=\frac{1}{2}(e^s+e^{-s}  )w+\frac{1}{2} (\frac{2}{e^s+e^{-s}  }    )\frac{f\bar{w}}{|w|^2 }, \quad e^{-2s}-e^{2s}=\frac{2|f|}{|w|^2 }.
\]
\end{lem}

The above formulae are quite tedious to work with. But in the region $\{ |f|\leq \epsilon H \}$, where $\epsilon$ is small, then $s$ is almost zero, and the above formulae are well approximated by a much simpler parametrisation over the base $\C$:
\begin{equation}
\chi: X_0\times \C \to \C^3, \quad z=w+\frac{f\bar{w}}{2|w|^2 }.
\end{equation}

The region $\{ |f|\leq \epsilon H, H>>1 \}$ corresponds roughly to $\{ |f|\leq \epsilon |w|^2, |w|^2>>1 \}$.
We notice $\chi$ is not exactly holomorphic, but only approximately so up to an error of order $O(\epsilon)$. Another caveat is that $w\in X_0$, so in coordinates $\sum w_i^2=0$.

We now study the asymptotic formula of the metric tensor associated to $\omega$ via the parametrisation $\chi$. The base term $\frac{1}{2}\sqrt{-1}df\wedge d\bar{f}$ pulls back exactly to $\frac{1}{2}\sqrt{-1}df\wedge d\bar{f}$, because $\chi$ is defined over the base. The function $H$ pulls back to $|w|^2 (1+O(\epsilon))$. To calculate the other terms we need to be careful about the failure of holomorphicity. This means we need to directly pull back the Riemmanian metric tensor. 

The term $\sqrt{-1}\ddbar{H}$ defines a two-tensor $\sum dz_i \otimes d\bar{z_i}=\sum d\bar{z_i} \otimes d{z_i}$, where to save writing we have passed to the symmetrization. 
The Riemannian metric tensor is two times its real part. Via $\chi$, we can write $dz_i=dw_i+df \frac{w_i}{2|w|^2}+\frac{f}{2|w|^2}d\bar{w_i} -\frac{f\bar{w_i}}{2|w|^4 }d|w|^2
$. We compute
\[
\begin{split}
Re(\sum dz_i \otimes d\bar{z_i})=& Re \{(\sum dw_i\otimes d\bar{w_i}) (1+ \frac{|f|^2}{ 4|w|^4}) \\
 + & (\frac{\bar{f} }{|w|^2}d{w_i}\otimes dw_i  -\frac{f}{ 2|w|^4 } d\bar{f} \otimes \bar{w_i} dw_i ) +\frac{df\otimes d\bar{f} }{4|w|^2 } \}.
\end{split}
\]
Up to $O(\epsilon)$ relative error, this can be simplified as 
\[
Re \{\sum dw_i\otimes d\bar{w_i}     -\frac{f}{ 2|w|^4 } d\bar{f} \otimes \bar{w_i} dw_i  +\frac{df\otimes d\bar{f} }{4|w|^2 } \}.
\]
In the presence of the base metric term $ Re (df\otimes d\bar{f} )$, we can simplify further to get the asymptote up to $O(\epsilon)$ error:
\[
H^{-1/2} Re \sum dz_i \otimes d\bar{z_i} \sim |w|^{-1} Re\sum dw_i\otimes d\bar{w_i}.
\]
Similarly the term $\sqrt{-1}\partial{H}\wedge \bar{\partial} H$ defines a symmetric two-tensor $\sum \bar{z_i}dz_i\otimes z_jd\bar{z_j}$. We have the asymptote up to $O(\epsilon)$ error
\[
H^{-3/2} Re\sum \bar{z_i}dz_i\otimes z_jd\bar{z_j} \sim |w|^{-3} Re\sum \bar{w_i}dw_i\otimes w_jd\bar{w_j}.
\]
 To summarize,
 
 \begin{lem}
The Riemmanian metric $g$ associated to $\omega$ admits the approximate formula in the region $\{ |f|\leq \epsilon H, H>>1 \}$:
\begin{equation}
\chi^*g=  Re ( df\otimes d\bar{f}+ |w|^{-1}\sum dw_i\otimes d\bar{w_i}-\frac{1}{2}  |w|^{-3} \sum \bar{w_i}dw_i\otimes w_jd\bar{w_j}                   )(1+O(\epsilon))
\end{equation}
\ie up to relative error of order $O(\epsilon)$, the metric $g$ is asymptotic to the product metric $g_{product}$ on $X_0\times \C$, where the $X_0$ factor is equipped with the Stenzel metric.
 \end{lem}

The information at our disposal now allows us to identify the tangent cone at infinity for the metric $g$ associated to $\omega$.

\begin{prop}
The tangent cone at infinity is isometrically $X_0\times \C$ with the product metric.
\end{prop}

\begin{proof}
Clearly only asymptotic properties of the metric are relevant. Fix a small number $\epsilon$. Consider large values of $r$, and apply the scaling factor $\frac{1}{r^2}$ to the metric $g$. In the ball $B_{\frac{1}{r^2}g} (0, D )$ of fixed radius $D$, if there is a point $P$ in the region $E_\epsilon$, then by the lemma \ref{squash},  its distance to $\C^3\setminus E_\epsilon$ is controlled above by $\frac{CD^{1/4} }{\epsilon^{1/4} r^{3/4} } <C\epsilon$ for large $r$. Thus the only important contribution comes from the region $(\C^3\setminus E_\epsilon ) \cap \{H\geq C\}$. Up to $O(\epsilon)$ error, $(\C^3\setminus E_\epsilon) \cap B_{\frac{1}{r^2} g  }(0,D)     $ is Gromov-Hausdorff close to the metric space
\[
( B_{g_{product}}(0, rD) \cap  \{ |f|\leq \epsilon |w|^2   \}             , \frac{1}{r^2}g_{product}  ) \subset X_0\times \C. \]
Now we observe that $X_0\times \C$ with the product metric is conical, and scaling features imply that for large $r$ the region $\{|f|\geq \epsilon |w|^2 \}$ is unimportant. Thus the metric space is Gromov-Hausdorff close to $( B_{g_{product}}(0, rD)             , \frac{1}{r^2}g_{product}  )  $, which is the same as $(B_{g_{product}}(0, D)             , g_{product}  )     $. From this we see that as $r\to \infty$, $B_{\frac{1}{r^2}g} (0, D )$ Gromov-Hausdorff converges to $(B_{g_{product}}(0, D)             , g_{product}  )     $, for any fixed $D$, so the tangent cone is indeed $X_0\times \C$ with the product metric.
\end{proof}

\begin{rmk}
(Jumping of complex structure) The complex structure on the tangent cone differs from the original manifold $\C^3$. This can be explained by the flat family of subvarieties of $\C_y\times \C^3_{z_1,z_2,z_3}$, defined as
\[
\{  (y,z_1,z_2,z_3) | ty=z_1^2+z_2^2+z_3^2            \},
\]
whose generic fibre is just a graph over $\C^3$, and the central fibre is $X_0\times \C$.
\end{rmk}

\begin{rmk}(Measure theoretic paradox)\label{measuretheoryparadox}
The link of the tangent cone is singular. The singularity corresponds to the node in $X_0$, and the cone $E_\epsilon$ collapses to the nodal line in $X_0\times \C$. This behaviour is counterintuitive as $E_\epsilon$ seems to carry a large amount of Lebesgue measure, which disappears in the limit, but general theory of noncollapsing Gromov-Hausdorff convergence implies measure theoretic convergence. But this is not a contradiction because in a large ball with respect to the metric $g$, the set $E_\epsilon$ only occupies a very small portion of the measure.	

H-J. Hein pointed out an analogy with the Taub-NUT metric which may help to understand this measure theoretic picture. By a well known observation of LeBrun \cite{LeBrun}, the Taub-NUT space is biholomorphic to $\C^2$ with the standard holomorphic volume form. In fact the holomorphic symplectic moment map makes $\C^2$ into a fibration of affine quadrics, exactly analogous to our situation. The counterintuitive fact then is that the Taub-NUT metric has cubic volume growth. The explanation again lies in the large distortion of distance. The distinctive feature of $n=2$ is that the quadric fibres are isomorphic to $\C^*$, whose natural cylindrical metric has linear growth rate, which is not maximal. The Taub-NUT metric is asymptotically described by the semiflat metric coming from the quadric fibration; the growth rate is quadratic in the base direction and linear in the fibre direction, thus the overall growth is cubic.
\end{rmk}

\subsection{Topological discussions}\label{topologicaldiscussions}

The topological situation is related to the general story of Milnor fibrations associated to hypersurface singularities. Let us dispense with the metric geometry for a moment. In the ODP context, the sphere $S^5=\{H=1\}$ intersects $\C^3\setminus X_0$ in an open set, which is just the complement of the link $X_0\cap S^5$ of $X_0$. We have a fibration over the circle
\[
S^5\cap ( \C^3\setminus X_0       )\to S^1, \quad z\mapsto \frac{f}{|f|}.
\]
The fibre is topologically identified with the Milnor fibre $X_1$. This is because via the radial projection map, this fibration is identified with 
\[
f: f^{-1}(S^1)\subset \C^3 \to S^1.
\]
The monodromy around the circle is described by the standard Picard-Lefschetz formula.

The analogous picture on the tangent cone $X_0\times \C$ is the following. We have a fibration
\[
\{ |w|+ |f|^2=1 \} \cap (X_0\times \C \setminus X_0\times\{0\} ) \to S^1, \quad (w,f)\mapsto \frac{f}{|f|}.
\]
The fibre is homeomorphic to $X_0\simeq \C^2/\Z_2$. The effect of the metric geometry is to collapse the vanishing cycle to a point, thereby changing $X_1$ to $X_0$. 

More explicitly, in our metric $g$, the vanishing cycle $\{ H=|f|\}$ on $X_f$ for large $|f|$ has diameter proportional to $|f|^{1/4}$, but the distance of the vanishing cycle to the origin is of order $|f|$. The tangent cone only sees the regions whose diameter has at least linear growth rate, so the vanishing cycle disappears in the limit. This phenomenon also appears in the announced work of Hein and Naber on $A_k$ type singularities.

The link of the tangent cone $X_0\times \C\simeq \C^2/\Z_2 \times \C$ is $S( \C^2/\Z_2 \times \C  )$, which is singular along a circle $S^1$. The appearance of the circle is explained by the $U(1)$ symmetry of the whole setup. Metrically, the points on the circle correspond to the Gromov Hausdorff limit of the vanishing cycles.

It is also interesting to make some comparisons of our situation with Joyce's QALE spaces (\cf \cite{Joyce}). The main common feature is that the tangent cone at infinity is a finite group quotient $\C^n/G=\C^3/\Z_2$, where the main complication comes from the non-free action of $\Z_2$. However, our setup does not seem to fit naturally into Joyce's framework, not even at the topological level. Joyce requires the ambient manifold to be a resolution of the quotient $\C^n/G$, which is not a natural structure on $\C^3$. Joyce understands $\C^n/G$ in terms of a stratification by fixed point sets of subgroups of $G$, and describes the structure of his resolution in a stratified manner in terms of what he calls the local product resolution; the simplest local picture looks like 
\[
Y\times \C^m \to \C^{n-m}/H \times \C^m,
\]
where $Y$ is a resolution of $\C^{n-m}/H$ and $H$ is a subgroup of $G$. Here $Y$ plays a similar role as $X_1$.
In my opinion his framework does not typically incorporate the Picard-Lefschetz monodromy behaviour, as his product structures are usually globally trivial over stratified pieces. At the metric level, on the local piece $Y\times \C$, Joyce requires asymptotic convergence of the metric to the product metric. In our case the vanishing cycles do not have constant size, but instead grow slowly as $O(a^{1/4})$ with the distance to the origin, so the metric in $E_\epsilon$ cannot converge to a product metric.

\subsection{The Ricci curvature}\label{TheRiccicurvatureofthemetricansatz}

Having examined the leading order behaviour of the metric, let us compute the Ricci tensor for $H>>1$. The formula for the Ricci form is 
\[
\text{Ric}=-\sqrt{-1}\ddbar{\\ \log \det(g_{i \bar{j} } )   }.
\]
Suppose we write $v=\det(g_{i \bar{j} } )$ for the volume density function. Then 
\[
\text{Ric}= \frac{ \sqrt{-1}\partial{v}\wedge \bar{\partial}v   }{v^2} -\frac{\sqrt{-1} \ddbar{v} }{v}.
\]
The nature of our ansatz means $v$ depends only on $H,\eta$. We have
\[
\begin{split}
\sqrt{-1}\partial\bar{\partial} v=& v_H\sqrt{-1} \partial\bar{\partial}H+v_{HH}\sqrt{-1} \partial{H}\wedge \bar{\partial} H
+\sqrt{-1} v_{H\eta} (\bar{f} df\wedge \bar{\partial}H +f \partial{H}  \wedge d\bar{f}         )\\
&+
(\eta v_{\eta\eta}  +v_\eta )\sqrt{-1}df d\bar{f}   .
\end{split}
\]
\[
\sqrt{-1}\partial{v}\wedge \bar{\partial}v =
v_H^2 \sqrt{-1}\partial{H}\wedge \bar{\partial}H
+\sqrt{-1} v_H v_\eta (  \partial{H}\wedge f d\bar{f}+ \bar{f} df \wedge \bar{\partial}{H}     ) +v_\eta^2 \eta \sqrt{-1} df\wedge d\bar{f}.
\]
The asymptotic expression is given by (\ref{asymptoticvolume1}).
The detailed expression of $v$ is very complicated, but there are some general features of its derivatives which can be used to our advantage. Everytime we differentiate with respect to $H$, we gain a suppression factor $O(\frac{1}{H})$; everytime we differentiate with respect to $\eta$, we gain $O(\frac{1}{a^2})$. To leading order, $v$ is a constant.

 We can use the above principle to estimate $\text{Ric}$. The worst term in (\ref{asymptoticvolume1}) is $\frac{1}{H^{1/2}a}$. It is clear by computation that the terms showing up in the quadratic expression $\sqrt{-1}\partial{v}\wedge \bar{\partial}v$ are much smaller than the terms in the second derivative $\sqrt{-1}\partial\bar{\partial} v$. This reflects the fact that the approximation is good enough for nonlinear effect to be insignificant in the asymptotic regime. Comparing the computation with the asymptotic formula
 \[
 \omega\sim \frac{1}{2}df\wedge d\bar{f}+ \frac{\sqrt{-1} }{ 2\sqrt{H+a}    }\ddbar{H}- \frac{\sqrt{-1} }{ 4 (\sqrt{H+a})^3    }\partial{H}\wedge \bar{\partial} H ,
 \]
 we see $  -\frac{C}{aH}\omega    \leq  \text{ Ric } \leq \frac{C}{aH}\omega$. Thus the Ricci tensor is controlled by the metric:
 \[
 |\text{Ricci tensor}| \leq \frac{C}{aH}g. 
 \]
 This indicates several features of the approximate ansatz. First, recall the distance to the origin is comparable to $a$. In the metric $g$, a typical point has $H\sim a^4$, so the Ricci tensor decays as $O(\frac{1}{a^5})$, which is quite fast. Second, in the worst region $E_\epsilon$, $H$ is comparable to $a$, so this ansatz only gives quadratic decay of the Ricci curvature. The scalar curvature likewise decays as $O(\frac{1}{a^2})$ in the region $E_\epsilon$.
 
 The last observation links this discussion to the work of Hein \cite{Hein1} on weighted Sobolev inequalities, and their applications to Monge-Amp\`ere equations (\cf also Hein's PhD thesis \cite{Hein2}). The main technical condition there is 
 \begin{Def}
 A complete noncompact Riemannian manifold $(M,g)$ is called $SOB(\beta)$ if there is a point $x_0\in M$ and $C\geq 1$, such that the annulus $A(x_0, s,t)$ is connected for all $t>s\geq C$, $|B(x_0,s)|\leq Cs^\beta$ for all $s\geq C$, and $|B(x, (1-\frac{1}{C}) r(x)   )|\geq \frac{1}{C} r(x)^\beta$, and $\text{Ricci} \geq -C r(x)^{-2} $ if $r(x)= d(x_0, x) \geq C$.
 \end{Def}
 
 \begin{prop}
 The metric ansatz (\ref{smoothingexample2scaled}) satisifies the condition $SOB(6)$.
 \end{prop}
 
 \begin{proof}
 The main point is that we know the maximal volume growth and the at worst quadratic Ricci curvature decay. The completeness property is clear from the asymptotic formula.
 \end{proof}
 
 We take the opportunity to explain another technical condition in Hein's work.
 \begin{Def}(\cf \cite{Hein2}, Definition 4.2)
 A $C^{k,\alpha}$ quasi-atlas on a complete K\"ahler manifold $(M,\omega)$ is a collection $\Phi_x: x\in A, A\subset M$ of holomorphic local diffeomorphisms $\Phi_x: B\to M, \Phi_x(0)=x,$ from $B=B(0,1)\subset \C^n$ into $M$ which extend smoothly to the closure $\bar{B}$, and such that there exists $C\geq 1$ with $\text{inj}(\Phi_x^*g) \geq 1/C$, $\frac{1}{C}g_{\C^n} \leq \Phi_x^*g\leq Cg_{\C^n}$, and $\norm{ \Phi_x^*g}_{C^{k,\alpha}(B, g_{\C^n}) } \leq C$ for all $x\in A$, and such that for all $y\in M$ there exists $x\in A$ with $y\in \Phi_x(B)$ and $\text{dist}_g (y, \partial \Phi_x(B)) \geq 1/C$.
 \end{Def}
 
 \begin{lem}\label{quasiatlas}
 Our metric ansatz (\ref{smoothingexample2scaled}) has a $C^{k, \alpha}$ quasi-atlas for any $k$.
 \end{lem}
 
 \begin{proof}
 It is enough to consider $H>>1$. For any such point $x=(x_1,x_2,x_3)\in \C^3$, without loss of generality $|x_1|=\max |x_i|$. Then we take the local coordinates $z_1'= f-f(x), z_2'= \frac{1}{H(x)^{1/4}}   (z_2-x_2), z_3'=\frac{1}{H(x)^{1/4}}  (z_3-x_3) $. Using the asymptotic formulae for the metric $g$, we see that in a coordinate ball the metric $g$ is uniformly equivalent to the Euclidean metric. In fact the uniform equivalence holds up to distance scale $O(H^{1/4})$. The higher order behaviour is clear.
 \end{proof}

\subsection{The asymptotic Laplacian}

It is important in the study of the Monge-Amp\`ere equation to have a well controlled linear theory of the Laplace operator. We want a formula for the main terms in the Laplacian of a general function $b(H,\eta)$ depending only on $H,\eta$. One quick method is to take the first variation of our formula for the volume form (\ref{symmetricansatz}). We record the result of the computation.

\begin{lem}
For the metric (\ref{smoothingexample2scaled}),
the formula for the Laplacian is 
\begin{equation}\label{Laplacianformula}
\begin{split}
(\sqrt{-1}\ddbar{b}) (\sqrt{-1}\ddbar{\phi}   )^2=\widetilde{vol_E}\{
&( \frac{3H+a}{ \sqrt{H+a}   } +O(\frac{1}{a} )              ) b_H  
+ (\frac{2(H^2-a^2)}{\sqrt{H+a}  } +O(\frac{H}{a} )   ) b_{HH} \\
+& (  O(\frac{\eta}{a}   )          ) b_{H\eta} 
+ (1+O (\frac{1}{ aH^{1/2}})  ) (\eta b_{\eta\eta}+b_\eta)
\}.
\end{split}
\end{equation}
\end{lem}

The interesting feature of this expression is its inhomogeneity. Think about the region $E_\epsilon$, where $|f|$ and $H$ are comparable. If we take a homogeneous expression $b$, such as $(|f|+H)^\alpha$, then each $H$ differentiation brings down the degree by 2, and each $\eta$ differentiation brings down the degree by 4. Then purely on the ground of dimensional analysis, only the terms with $b_H$ and $b_{HH}$ dominate. 
There is an additional subtlety that $H^2-a^2$ can have zeros. If we ignore this issue, then the insight we gain from this formula is that if we just want to invert the Laplacian approximately in the region $E_\epsilon \cap \{H>>1 \}$, then we only need to solve a second order ODE rather than a PDE. The fact that the $\eta$ derivative is suppressed, can be interpreted as saying the fibres behave as if they were independent in the region $E_\epsilon$. This phenomenon is typical in adiabatic limit problems, where the metric ansatz ultimately comes from.

One explanation for this phenomenon is that in the intrinsic geometry of the metric $g$, one should trade $|f|$ for $H^{1/4}$. With this new homogeneity convention, in the region where $|f|\sim H^{1/4}$, the quantity $H^{1/2 }b_H$ should be comparable to $b_\eta$, so the Laplacian genuinely depends on two variables in that region. The nonclassical nature of the Laplacian in the region $E_\epsilon$ testifies to the singularity of the tangent cone, which breaks down most of the intuitions based on Euclidean geometry.

\section{Calabi-Yau metric}

We aim to improve the approximate solution into a genuine Calabi-Yau metric. Our strategy to achieve this relies on the following existence result for the Monge-Amp\`ere equation on noncompact manifolds, taken from Hein's PhD thesis \cite{Hein2}, which builds on previous work of Tian and Yau \cite{TianYau}. 

\begin{thm}
(\cite{Hein2}, proposition 4.1) Let $(M, \omega)$ be a complete noncompact K\"ahler manifold with a $C^{3,\alpha}$ atlas, which satisfies the condition $SOB(\beta)$, where $\beta>2$.
Let $f\in C^{2,\alpha}(M)$ satisfy $|f|\leq Cr^{-\mu}$ on $\{ r>1 \}$ for some $\mu>2$. Then there exists $\bar{\alpha} \in (0,\alpha]  $ and $u\in C^{4,\bar{\alpha} }(M)$, such that $(\omega+i\ddbar{u})^n=e^f \omega^n$. If in addition $f\in C^{k,\bar{\alpha} }(M)$ for some $k\geq 3$, then all such solutions $u$ belong to $C^{k+2, \bar{\alpha}}_{\text{loc} } (M)$ with estimates. 
\end{thm}

\begin{rmk}
The proof considers a family of viscosity equations of the shape
\[
(\omega+ i \ddbar{u_\epsilon})^n= e^{f+\epsilon u_\epsilon} \omega^n,
\]
and derives uniform $L^\infty$ estimate for all small $\epsilon>0$. The desired solution follows by taking a subsequential limit, which is not a priori known to be unique.  

The condition $SOB(\beta)$ is chiefly used in the application of certain weighted Sobolev inequalities. The condition $\beta>2$ reflects a basic distinction between parabolic and non-parabolic manifolds. The reason to impose $\mu>2$ is that the method typically only allows us to construct solutions with decaying potential, and if the solution decays with some power law $O(r^{2-\mu})$, then standard potential theory on Euclidean spaces makes us expect the Laplacian to decay with order $O(r^{-\mu})$. As a side remark, since the distance function in our ansatz involves two competing factors $|f|$ and $H^{1/4}$, the intuition from standard Euclidean behaviour is not always very effective.
\end{rmk}

\subsection{The error terms for the volume of the ansatz (\ref{smoothingexample2scaled})}

We wish to understand accurately the error terms in our ansatz (\ref{smoothingexample2scaled}). To aid computation, let us introduce the quantities
\[
Q_1=F_H^3+F_H^2 H F_{HH}, \quad Q_2=\eta F_{\eta\eta}+F_\eta,
\]
\[
Q_3=HF_H^2+(H^2-\eta) F_H F_{HH}, \quad Q_4=F_{H\eta} F_H^2\eta-F_{H\eta}^2 \eta F_H (H^2-\eta).
\]
Our general formula for the volume form reads
\[
\frac{ (\sqrt{-1}\ddbar{\phi})^3   }{ 6 \widetilde{vol_E}  }=Q_1+4Q_2Q_3+4Q_4.
\]
In the light of Hein's theorem, we shall keep only the terms with slower than quadratic decay. Recall the distance to the origin is of the same order as the function $a$. We state the result of a long computation:
\[
Q_1=-\frac{H}{16 (H+a)^{5/2}  }+\frac{1}{8 (H+a)^{3/2} }+ O(  \frac{1}{H^2 a} ),
\]
\[
Q_2=\frac{1}{2} + \frac{1}{8 a\sqrt{H+a}  }-\frac{1}{16 (H+a)^{3/2}   }+O(\frac{1}{a^3}),
\]
\[
Q_3=\frac{1}{8}+ \frac{ H^{1/2} } { 32a(H+a)   } -\frac{ H^{1/2} }{ 8( H+a  )^2  } + \frac{3}{ 32(H+a) \sqrt{H}  } -\frac{H-a} {32 aH^{3/2}} +O(\frac{1}{a^3}),
\]
\[
Q_4=-\frac{1}{32} \frac{a}{(H+a)^{5/2}  }-\frac{ H-a} { 128 (H+a)^{5/2}   }+O(\frac{1}{a^3}).
\]

\begin{rmk}
In the region where $H$ is comparable to $a^4$,  all error terms are of order at most $O(\frac{1}{a^3})$, so the approximation is already good enough. But in the region $E_\epsilon$, we do need to keep many error terms. As observed earlier, the worst error term has order $O(\frac{1}{aH^{1/2}})$, so its square is of order $O(\frac{1}{a^2 H})$, which indicates that nonlinear effects are asymptotically unimportant.
\end{rmk}

From these computations, we see that
\begin{lem} 
The volume form of the metric given by the ansatz (\ref{smoothingexample2scaled}) is
\begin{equation}\label{volumeerrorterms}
\begin{split}
\frac{ (\sqrt{-1}\ddbar{\phi})^3   }{ 6 \widetilde{vol_E}  }=\frac{1}{4} \{  
1+ \frac{1}{4a(H+a)^{1/2} }+ \frac{H^{1/2}} { 4a (H+a) }
-\frac{H^{1/2} }{ (H+a)^2   } \\
+\frac{3}{ 4(H+a) \sqrt{H}     }
-\frac{H-a}{   4aH^{3/2}        }
\}
+O(\frac{1}{a^3} ).
\end{split}
\end{equation}
\end{lem}
This can be viewed as an improved version of (\ref{asymptoticvolume1}).

\subsection{Approximately inverting the Laplacian}

The problem of applying Hein's result to construct a genuine Calabi-Yau metric is that the main error terms have slower than quadratic decay. We shall attempt to remove them by explicitly inverting the Laplacian up to an admissible amount of error. The result is
\begin{lem}\label{correctmainerrors}
	There exists a function $b=\frac{1}{a} \tilde{b} (\frac{H}{a})$ defined for $H>>1$, where $\tilde{b}$ is smooth with growth control \[\tilde{b}(x)=O(\log x  ), \quad  \frac{d^k\tilde{b} }{dx^k}=O(x^{-k}),\] such that
	\[
	\begin{split}
	(\sqrt{-1}\ddbar{b}) (\sqrt{-1}\ddbar{\phi}   )^2
	= \frac{1}{2} \widetilde{vol_E}   \{  
	\frac{1}{4a(H+a)^{1/2} }+ \frac{H^{1/2}} { 4a (H+a) }
	-\frac{H^{1/2} }{ (H+a)^2   } \\
	+\frac{3}{ 4(H+a) \sqrt{H}     }
	-\frac{H-a}{   4aH^{3/2}        }
	\}
	+O(\frac{ \log (2H/a)   }{a^3} ) \}.
	\end{split}
	\] 
\end{lem}

We begin with examples of treating individual terms before describing the general algorithm.

\begin{eg}
We treat the error term $\frac{1}{4a(H+a)^{1/2}}$ in detail. The leading order terms for the Laplacian of a function $b(H, \eta)$ is
\[
(\sqrt{-1}\ddbar{b}) (\sqrt{-1}\ddbar{\phi}   )^2 \sim \widetilde{vol_E}\{
 \frac{3H+a}{ \sqrt{H+a}   }                b_H  
+ \frac{2(H^2-a^2)}{\sqrt{H+a}  } b_{HH} \}
\]
Now we solve a related ODE with variable $x$:
\[
\frac{3x+y}{ \sqrt{ x+y } } h'+ \frac{2(x^2-y^2)}{\sqrt{x+y}  }h''= \frac{1}{4 y(x+y)^{1/2}  }.
\]
We can rewrite the ODE as
\[
\frac{d}{dx} \{     2(x-y)\sqrt{x+y} \frac{dh}{dx}           \}=\frac{1}{4 y(x+y)^{1/2}  }.
\]
This is a direct consequence of the asymptotic Laplacian formula, from which it is clear that the ODE is reducible to elementary integration.
The first integration gives
\[
\frac{dh}{dx}= \frac{1}{4y(x-y)\sqrt{x+y}   }(  {\sqrt{ x+y   }} -\sqrt{2y}  ) ,
\]
where we choose the normalisation that for $x=y$ the derivative $\frac{dh}{dx}$ remains bounded. This is to ensure the smooth dependence of $h$ on $x$. Integrate further to get
\[
h(x)= \frac{1}{2y}  \log ( \frac{ \sqrt{x+y}+ \sqrt{2y}  }{ \sqrt{y} } ) .
\]
Here the normalisation is chosen to display homogeneity behaviour.

We then substitute $H$ for $x$ and $a$ for $y$. This gives an approximate solution
\[
b_1(H,\eta)=\frac{1}{2a  }\log ( \frac{ \sqrt{H+a}+ \sqrt{2a}  }{\sqrt{a} } ).
\]
By construction, the expression 
\[
\frac{3H+a}{ \sqrt{H+a}   }                b_{1H}  
+ \frac{2(H^2-a^2)}{\sqrt{H+a}  } b_{1HH} = \frac{1}{4a(H+a)^{1/2}} + O( \frac{\log (\frac{2H}{a})}{a^3}  )
\]
The error term $ O( \frac{\log (\frac{2H}{a})}{a^3}  )$ appears because $a$ also depends on $H$, but the relative error is quite small, so the situation is improved. Now we use the formula (\ref{Laplacianformula}) to see that in fact 
\[
(\sqrt{-1}\ddbar{b_1}) (\sqrt{-1}\ddbar{\phi}   )^2 =\widetilde{vol_E}\{ \frac{1}{4a(H+a)^{1/2}} + O(\frac{\log (2H/a)}{a^3}) 
 \}.
\]
This means we can invert the Laplacian on this error term approximately, as desired. 

\begin{rmk}
We comment that the ability to solve for a decaying potential on a forcing term with slower than quadratic decay is not what one might expect from Euclidean potential theory.	One explanation for this, is that our forcing term is $ O(\frac{1}{a^3}  )$ in the region where $H\sim |f|^4$, and the slow decay occurs only in a very small solid angle.
\end{rmk}

\end{eg}

\begin{eg}
As a second example, we treat the term $\frac{1}{4H^{3/2}}$. We first write down the associated ODE:
\[
\frac{d}{dx} \{     2(x-y)\sqrt{x+y} \frac{dh}{dx}           \}=
\frac{3x+y}{ \sqrt{ x+y } } h'+ \frac{2(x^2-y^2)}{\sqrt{x+y}  }h''= \frac{1}{4 x^{3/2}  }.
\]
The first integration gives
\[
 \frac{dh}{dx}= \frac{1}{ 4(x-y)\sqrt{x+y}   } (\frac{1}{y^{1/2}} -\frac{1}{x^{1/2}} ),
\]
where the normalisation is chosen to make $\frac{dh}{dx}$ bounded at $y=x$. The RHS clearly exhibits homogeneity features, so we can find an integral in the form
\[
h(x,y)= \frac{1}{y} \tilde{h} ( \frac{x}{y} ),
\]
where 
\[
 \frac{d \tilde{h} }{dx}= \frac{1}{ 4(x-1)\sqrt{x+1}   } (1 -\frac{1}{x^{1/2}} ).
\]
The detailed expression of the RHS is less important than its qualitative features, namely that it is a smooth expression in $x$, and its asymptotic growth at large $x$ is $O(\frac{1}{x^{3/2}})$. This implies $\tilde{h}(x)= O(\frac{1}{x^{1/2}})$. We then define
\[
b_2(H, \eta)= \frac{1}{a} \tilde{h}( \frac{H}{a}   ),
\]
so by construction
\[
\frac{3H+a}{ \sqrt{H+a}   }                b_{2H}  
+ \frac{2(H^2-a^2)}{\sqrt{H+a}  } b_{2HH} = \frac{1}{4H^{3/2}} + O( \frac{ 1 }{a^{5/2} H^{1/2} }  ).
\]
Using the Laplacian formula (\ref{Laplacianformula}), one can see further that 
\[
(\sqrt{-1}\ddbar{b_2}) (\sqrt{-1}\ddbar{\phi}   )^2 =\widetilde{vol_E}\{ \frac{1}{4H^{3/2}} + O(\frac{1}{a^{5/2}H^{1/2} }) 
\}.
\]
This means the singular term can be approximately cancelled in this example.
\end{eg}

The features of these examples mostly carry over to the other error terms in (\ref{volumeerrorterms}) with minor changes. We can write down with obvious modifications an ODE for each error term. Explicit integrals in most other cases are more difficult to compute. However, by the homogeneity of the ODE, we can always write 
\[
h(x,y)=\frac{1}{y} \tilde{h} (\frac{x}{y}   ),
\]
where $\tilde{h}$ solves the ODE with parameter $y=1$. The algorithm is that we perform the first integral with the requirement that $\frac{d\tilde{h}}{dx}$ stays bounded near $x=1$, integrate to find $h$, substitute $H$ for $x$ and $a$ for $y$. It is clear that in each case the result is a smooth function $b(H, \eta)$ for $H>>1$. 

It remains to control the behaviour of the function $\tilde{h}(x)$ for large $x$, in order to estimate all the errors. For the terms
\[
\frac{1}{4a(H+a)^{1/2} }, \quad \frac{H^{1/2}} { 4a (H+a) }, \quad -\frac{1}{   4aH^{1/2}        },
\]
the relevant $\tilde{h}$ has leading order growth (up to constant factors)
$
\tilde{h} \sim \log x, \quad \frac{d\tilde{h} }{dx} \sim x^{-1}.
$
For the terms 
\[
-\frac{H^{1/2} }{ (H+a)^2   }, \quad
\frac{3}{ 4(H+a) \sqrt{H}     }, \quad
\frac{1}{   4H^{3/2}        },
\]
$\tilde{h}$ has leading order growth (up to constant factors) 
$
\tilde{h} \sim x^{-{1/2}}, \quad \frac{d\tilde{h} }{dx} \sim x^{-{3/2}}.
$

It follows that in the first three cases, we can cancel the error terms up to $O(\frac{\log (H/a)}{a^3})$, and in the last three cases, we can cancel the error terms up to $O(\frac{1}{H^{1/2}a^{5/2}  }  )$. In summary, we have the lemma \ref{correctmainerrors}.

\begin{rmk}
It should be pointed out that this ODE method improves the error only in the bad region with $|f|>>H^{1/4}$, \ie the region collapsing to the singularity of the tangent cone. Such ODEs appear quite often in adiabatic limit problems, where one has a local fibration structure and one seeks an improved approximate solution by solving an ODE along the fibres. This is indeed what happens here.

In the region where $|f|$ is comparable to $H^{1/4}$, the Laplacian genuinely depends on both $H$ and $\eta$ derivatives, so ODE reduction is not possible in general, but the error there is already acceptably small.
\end{rmk}

\begin{rmk}
The order of $\sqrt{-1}\ddbar{b}$ compared to $\sqrt{-1}\ddbar{\phi}$ is $O( \frac{1}{aH^{1/2}} )+ O(\frac{\log (2H/a) }{a^3}   )$, so its square has order $O(\frac{1}{a^3})$. This again confirms the fact that the nonlinear effect is asymptotically weak.
\end{rmk}


\subsection{Perturbation into the Calabi-Yau metric}

We improve the metric ansatz to allow for the application of Hein's existence result. For this we define a metric 
\begin{equation}
\omega'=\sqrt{-1} \ddbar{ (\phi-b \chi )   }.
\end{equation}
where $b$ is defined in Lemma \ref{correctmainerrors}, $0\leq \chi\leq 1$ is a cutoff function, vanishes in the ball $\{ H\leq R^2  \}$, and equals to one outside $\{ H\leq 4R^2  \}$. If we choose the parameter $R$ large enough, we can ensure $\omega'$ is a genuine K\"ahler metric.

For $H>>1$, the volume form of this modified ansatz is 
\[
\begin{split}
\omega'^3&=( \sqrt{-1} \ddbar{ \phi    })^3 -3 ( \sqrt{-1} \ddbar{ \phi    })^2 \sqrt{-1} \ddbar{ b   }+O(  \frac{1}{ a^3} )  \\
&=\frac{3}{2} \widetilde{vol_E}   \{  
1+ O(\frac{ \log (2H/a)   }{a^3} ) 
\}.
\end{split}
\]
Crudely speaking, the error is now at most $O(\frac{1}{a^{3-\epsilon}})$ for any small $\epsilon>0$.

Now since $\omega'$ deviates from $\sqrt{-1}\ddbar{\phi}$ by only a small smooth perturbation, Hein's condition $SOB(6)$ and the existence of $C^{k,\alpha}$ quasi-atlas still hold for $\omega'$ (compare with section \ref{TheRiccicurvatureofthemetricansatz}). So Hein's theorem gives us

\begin{thm}
There is a smooth bounded potential $\psi$ on $\C^3$ with bounded derivatives to all orders, solving the Monge-Amp\`ere equation
\begin{equation}
(\omega'+\sqrt{-1}\ddbar{\psi})^3=\frac{3}{2} \widetilde{vol_E}.
\end{equation}
In particular the metric $\omega_{CY}=\omega'+\sqrt{-1}\ddbar{\psi}$ is Calabi-Yau, and is uniformly equivalent to $\omega'$.
\end{thm}

Since our setup is symmetric under the action of $SO(3,\R)\times U(1)$, and Hein's method works equivariantly with respect to the group action, we can assume the solution is symmetric.

\subsection{Decay estimates}

We now improve the decay rate of the potential $\psi$. This can be extracted from Hein's work, and the method is standard.
We observe first
\begin{lem}\label{ddbarcontrolofa}
	The function $a$ is comparable to the distance to the origin for $H>1$, and globally satisfies the estimate
	\[
	|da|+a|\sqrt{-1}\ddbar{a}| \leq C.
	\]
\end{lem}

\begin{prop}
The potential $\psi$ decays as $O( \frac{1}{ a^ {1-\epsilon} } )$ for any $\epsilon>0$.
\end{prop}

\begin{proof}
This follows from Hein's thesis \cite{Hein2}, the proof of proposition 3.6 (ii), and his comments on non-parabolic manifolds in his section 4.5. The main idea is weighted Moser iteration. The lemma above is technically required in his argument. The general conclusion is that if the forcing term decays as $O(a^{-\mu})$, for $2<\mu<6$, then the potential decays as $O(a^{2-\mu+\epsilon})$.
\end{proof}

\begin{rmk}
We think the decay rate estimate is not far from optimal because our correction term to the potential is $b=O(\frac{\log (2H/a)}{a}   )$. It may be of interest to compare this with Joyce's work \cite{Joyce}. As discussed in section \ref{topologicaldiscussions}, our setup does not technically fit into Joyce's definition of QALE manifolds. So we shall just make the crude remark that for Joyce's ALE manifold of real dimension $2n$, the decay rate of the metric is $O(a^{-2n})$, and the decay rate of the potential is $O(a^{2-2n})$, which in our context is $O(a^{-4})$. In any case it is clear that our decay rate is much slower.
\end{rmk}

\begin{prop}\label{decayerrorestimate}
The potential $\psi$ satisfies the decay bound 
\begin{equation}
\norm{ \nabla_{\omega'}^{(k)}\psi}_{L^\infty} \leq \frac{C(k)}{ a^{1-\epsilon}(H+1)^{k/4}  }.
\end{equation}
\end{prop}

\begin{proof}
By the smooth estimates in Hein's result, it is enough to consider $H>>1$ and prove the decay property.
Recall in our proof of the existence of the $C^{k,\alpha}$ quasi-atlas, we introduced a set of coordinates $z_1',z_2',z_3'$ around the point $x$ (see lemma \ref{quasiatlas}), which displays that the metric is smoothly equivalent to the standard Euclidean metric on a ball of size $cH^{1/4}(x) $, for some small fixed $c$.
Now we scale both the coordinates and the metric, writing $z_j''=\frac{z_j'}{ cH(x)^{1/4} }$, so $\frac{1}{H(x)^{1/2}} \omega'$ and $\frac{1}{H(x)^{1/2}}\omega_{CY}$ are uniformly equivalent to the standard Euclidean metric in the unit coordinate ball. The Calabi-Yau condition reads
\[
(\frac{1}{H(x)^{1/2}} \omega'+ \frac{1}{H(x)^{1/2}} \sqrt{-1}\ddbar\psi   )^3-(\frac{1}{H(x)^{1/2}} \omega')^3 =\frac{1}{H(x)^{3/2}} (\omega_{CY}^3-\omega'^3).
\]
The $k-$th derivatives of the forcing function $\frac{\omega'^3-\omega_{CY}^3}{\omega_{CY}^3}$ is of order $O(\frac{ \log (2H/a   )}{ a^3 })=O( \frac{1}{a^{1-\epsilon} H^{1/2 } } )$, for any $k\geq 0$. The potential $\frac{\psi}{H(x)^{1/2}}=O( \frac{1}{a^{1-\epsilon} H^{1/2 } } )$. By elliptic bootstrap, we obtain that the $k-$th derivative of $\frac{\psi}{H(x)^{1/2}}$ is estimated by $O(\frac{1}{a^{1-\epsilon} H^{1/2 } } )$. Reverting back to the $z_i'$ coordinates, this implies the claim.
\end{proof}

\begin{rmk}
In the region where $|f|$ is comparable to $ H^{1/4}$, the $k=2$ case of this estimate implies $\omega_{CY}-\omega=O(\frac{1}{ a^{3-\epsilon}}  )$, where we recall $\omega$ is the ansatz (\ref{smoothingexample2scaled}). 
The decay estimates are much slower in $E_\epsilon$, which can be expected from the singularity of the tangent cone at infinity.
\end{rmk}

As a consequence of the decay estimate \ref{decayerrorestimate} on the metric, the leading order asymptote of $\omega$ agrees with that of $\omega_{CY}$, hence
\begin{cor}
The asymptotic properties of $\omega$, including maximal volume growth and the tangent cone at infinity, are also true for $\omega_{CY}$.
\end{cor}

\begin{rmk}
This CY metric is a counterexample to a previous conjecture, which says any complete CY metric on $\C^n$ with maximal growth is flat (\cf \cite{Tian}, Remark 5.3).
\end{rmk}

\subsection{Riemannian submersion property}

We consider the trace $\Tr_{\omega_{CY}}\omega_0$, where $\omega_0=\frac{1}{2} \sqrt{-1} df d\bar{f}$ is the standard Euclidean metric on $\C$. We notice that for $H>>1$, the asymptotic behaviour of $\omega_{CY}$ implies $f$ is almost a Riemannian submersion:
\[
\Tr_{\omega_{CY}}\omega_0 \sim 1.
\]
For the region $H<<1$, this is no longer true due to quantising effects. But quite remarkably we have a one sided global estimate:

\begin{prop}
$\Tr_{\omega_{CY}} \omega_0 \leq 1.$
\end{prop}

\begin{proof}
We use the Chern-Lu formula for the fibration $f$, which in this case essentially reduces to the Bochner formula. Notice $\omega_{CY}$ is Ricci flat and $\omega_0$ is flat, so \[
\Lap_{\omega_{CY}} \log \Tr_{\omega_{CY}} \omega_0 \geq 0.
\]
The subharmonic function $\log \Tr_{\omega_{CY}} \omega_0$ tends to limiting value zero at infinity, hence the maximum principle gives the claim.
\end{proof}

\begin{rmk}
It may be noticed that this argument only depends on the Calabi-Yau condition and the leading asymptote.
\end{rmk}

\section{Remarks on moduli questions}

\subsection{Moduli questions and possible generalisations}

Our Calabi-Yau metric $\omega_{CY}$ is a nontrivial Ricci flat metric on $\C^3$. This is interesting in relation to the following fundamental question: on a given noncompact complex manifold with a global holomorphic volume form, can we classify the complete Calabi-Yau metrics? What if we impose maximal growth condition? How about the case of $\C^n$?

Our example indicates that the question involves considerable subtlety. The author thinks his construction can be extended to higher dimensional standard Lefschetz fibrations on $\C^n$ with $n>3$, using a similar strategy, but we do not pursue it here due to the computational difficulty.

In a more general vein, one may ask if we can build other examples using more complicated fibrations on $\C^n$, starting from some ansatz which desingularises a semiflat metric; optimistically this gives an abundant supply of nontrivial Ricci flat metrics on $\C^n$. Gabor Szekelyhidi informs the author after the completion of this paper that he is independently working on this more general problem.
The difficulty for this question is more severe, because we lose both the explicit nature of the Stenzel metric and the cohomogeneity-two property.

This question is also very relevant for understanding what happens near the (isolated) critical points of a very collapsed Calabi-Yau manifold fibred over some Riemann surface, assuming the singular fibre is a normal Calabi-Yau variety. The expectation is that there should be a complete Calabi-Yau metric on a model space fibred over $\C$, possibly biholomorphic to $\C^n$, whose asymptotic behaviour is a semiflat metric, such that after appropriate scaling it describes the collapsing metric near the critical points on the total space of the fibration. (The normality condition is imposed for the fibrewise Calabi-Yau metric to make sense; if we drop this condition, a prototype example is the Ooguri-Vafa metric, which is an incomplete metric modelling the collapsing elliptically fibred K3 surface near a singular fibre (\cf Gross and Wilson \cite{GrossWilson}), and whose local behaviour near the node is modelled on the Taub-NUT metric which does not have maximal growth.) A description of the tangent cone at infinity is likely to involve stability issues, as understood for $A_k$-type singularities by Hein and Naber.

On the negative side, any meaningful classification program must deal with singular tangent cones,
which can be a major source of difficulty.

\subsection{Deformation and uniqueness questions}

We can also ask the weaker question of classifying Calabi-Yau metrics arising as deformations of a given one, which I believe is more tractable. This depends on the class of deformations, which in turn depends on the function spaces in which one sets up the Monge-Amp\`ere equation. Hein's viscosity solution method surprisingly does not give any suggestion on the appropriate function space, nor does it make any uniqueness assertions. Let us formulate a number of more concrete questions:

\begin{Question}
Can we classify all Calabi-Yau metrics with the given holomorphic volume form which are uniformly equivalent to our example?
\end{Question}

Instead of giving a complete answer, we remark that  $SL(3,
\C)$ gives automorphisms of $\C^3$ as a complex manifold. But the uniform equivalence forces a compatibility condition with the fibration structure, so reduces this symmetry to $SO(3,\C) \times \langle e^{2\pi i/3}\rangle$. Clearly $SO(3,\R)\times \langle e^{2\pi i/3}\rangle$ is a subgroup preserving also the metric structure. So the moduli space contains at least the homogeneous space $SO(3,\C) / SO(3,\R) \simeq SL(2,\C)/SU(2) $. These examples are different as Calabi-Yau metrics on the fixed complex manifold $\C^3$, but they are isometric as Riemannian manifolds by a linear change of coordinates.

A more subtle deformation of our metric can be obtained by a simultaneous scaling of our metric and the coordinates, such that the volume form is preserved. Viewed alternatively, if we start from the ansatz
\[
\phi=\frac{1}{2\lambda^{2/3}}  \eta+  \lambda^{1/3} \sqrt{H+ ( \eta+ \lambda \sqrt{H+\lambda^{2/3}}   )^{1/2}  },
\]
where $\lambda>0$, and solve the Monge-Amp\`ere equation as before, we would get a new Calabi-Yau metric with the same volume form. This deformation can be anticipated from our discussions in section \ref{collapsingproblem}.

\begin{Question}
Given the leading order asymptote for the K\"ahler form, then is our Calabi-Yau metric unique?
\end{Question}

The uniqueness is clear once we assume sufficient decay of the perturbation term, and the question is about the optimal decay condition we need.

\begin{Question}
What are the formal deformations of our example?
\end{Question}

This involves linearising the Calabi-Yau equation. We obtain formally an equation on a closed $(1,1)$ form $\beta$:
\begin{equation}
\omega_{CY}^2 \beta=0.
\end{equation}
This says $\beta$ is primitive. On the space $\C^3$, closedness and exactness are the same. It is clear from the Hodge identities that $\beta$ is harmonic. For example, the nontrivial deformations discussed above provide particular solutions. The answer in general depends on the function space. In particular, we may ask what are the bounded solutions, and how they relate to moduli questions. 

\begin{Question}
(Rigidity) If $\omega_{CY}'$ is a Calabi-Yau metric on a large ball $\{|z|\leq R \}$ with the same volume form as $\omega_{CY}$, and near the boundary $ (1-\epsilon) \omega_{CY} \leq   \omega_{CY}' \leq (1+\epsilon) \omega_{CY}$ for sufficiently small $\epsilon$, then can we deduce effective estimates of the shape
$
(1-C\epsilon) \omega_{CY} \leq   \omega_{CY}' \leq (1+C\epsilon) \omega_{CY} 
$ on the whole ball, where $C$ is independent of $R$?
\end{Question}


\end{document}